\documentclass[11pt,reqno]{amsart}
 \usepackage[dvips]{epsfig}
 \usepackage{amsgen, amstext,amsbsy,amsopn, amsthm, amsfonts,amssymb,amscd,amsmat
 h,euscript,enumerate,url,verbatim,calc,xypic}
 \usepackage[colorlinks]{hyperref}
 \usepackage{latexsym}
 \usepackage{graphics}
 \usepackage{color}

 \newcommand{\m}{\mathfrak{m} }

\theoremstyle{plain}
 \newtheorem{thm}{Theorem}[section]
 \newtheorem{prop}[thm]{Proposition}
 \newtheorem{lem}[thm]{Lemma}
 \newtheorem{cor}[thm]{Corollary}
 
\theoremstyle{definition}

\theoremstyle{remark}
 \newtheorem{rem}[thm]{Remark}
 \newtheorem{question}{Question}

\newcommand{\Tor}{\operatorname{Tor}}

\newcommand{\Ker}{\operatorname{Ker}}

\newcommand{\Image}{\operatorname{Im}}

\newcommand{\Syz}{\operatorname{Syz}}

\newcommand{\Ho}{\operatorname{H}}

\newcommand{\lin}{\operatorname{lin}}
\newcommand{\ld}{\operatorname{ld}}

\newcommand{\agr}{\operatorname{\mathsf g}}

\title[ ] { Linearity defect of the residue field of short local rings}

\author{  Rasoul Ahangari Maleki }

\address{Rasoul Ahangari Maleki, School of Mathematics, Institute for Research in Fundamental Sciences (IPM), P.O. Box: 19395-5746, Tehran, Iran}
\email{rahangari@ipm.ir, rasoulahangari@gmail.com}


\thanks{This research was in part  supported by a grant from IPM (No. 94130028). This work was also supported by the Iran National Science Foundation (INSF) grant No. 95001343}
\numberwithin{equation}{thm}

\begin{document}
\maketitle
\begin{abstract}
Let  $(R,\m,k)$  be a Noetherian  local ring with  maximal ideal
$\m$ and  residue field $k$. The linearity  defect of a finitely
generated  $R$-module $M$, which is denoted  $\ld_R(M)$, is a
numerical measure of how far  $M$  is from having linear resolution.
We study the linearity defect of the residue field. We give a
positive answer to the question raised by Herzog and Iyengar of
whether $\ld_R(k)<\infty$ implies $\ld_R(k)=0$, in the case when
$\m^4=0$.
\end{abstract}



\vskip 1 cm

\smallskip
\section{Introduction and notation}

This paper is concerned with  the notion of the linearity defect of
the residue field  of a commutative Noetherian local ring. This
invariant was introduced by Herzog and Iyengar \cite{H-i} and has
been further studied by Iyengar and R\"{o}mer \cite {I-T}, \c{S}ega
\cite{S} and Nguyen \cite{Hop}. Let us recall the definition of the
linearity defect. Throughout this paper $(R,\m,k)$ will denote a
commutative Noetherian local ring with maximal ideal $\m$ and
residue field $k$. Let
$$ \mathbf{F}: \cdots \rightarrow F_{n}\xrightarrow{\partial_{n}}F_{n-1}\xrightarrow{\partial_{n-1}} \cdots\rightarrow F_{1}\xrightarrow{\partial_{1}} F_{0}\rightarrow 0$$
be a minimal complex (i.e. $\partial_i(F_i)\subseteq \m F_{i-1}$ for all $i\geq 0$) of finitely generated free $R$-modules.
Then  the complex
has a filtration $\{\mathfrak{F}^p \mathbf{F}\}_{p\geq
 0}$ with $(\mathfrak{F}^p \mathbf{F})_i=\m^{p-i}F_i$ for all $p$
 and $i$ where, by  convention,  $\m^j=R$ for all $j\leq 0$. The
 associated graded complex with respect to this filtration is called
  \textit{the linear part} of $ \mathbf{F}$ and denoted by
  $\lin^R(\mathbf{F})$. Let $N$ be an $R$-module. The notation $R^{\agr}$ will stand for
  the associated graded ring $\oplus_{i\geq 0}\m^i/\m^{i+1}$ and $N^{\agr}$ for the associated
  graded $R^{\agr}$-module $\oplus_{i\geq 0}\m^iN/\m^{i+1}N$. By
  construction,  $\lin^R(\mathbf{F})$ is a graded complex of graded
  free $R^{\agr}$-modules and has the property that
  $\lin^R_n(\mathbf{F})=F_n^{\agr}(-n)$, for all $n$. For more information about this
  complex, we again refer
  to \cite{H-i} and \cite {I-T}.
  Let $M$ be a finitely generated $R$-module. The \textit{linearity defect} of $M$ is
defined to be the number
$$\ld_{R}(M):=\sup\{i\in\mathbb{Z}|  \Ho_{i}(\lin^{R}(\mathbf{F})) \neq 0\},$$
where $\mathbf{F}$ is a minimal free resolution of $M$.
 By definition,  $\ld_R(M)$ can be infinite and
$\ld_R(M)\leq d$  is finite if and only if $(\Syz_d(M))^{\agr}$ has
a linear resolution over the standard graded algebra $R^{\agr}$,
where  $\Syz_d(M)$ is the $d$th syzygy module of $M$. In particular,
$\ld_R(M)=0$ if and only if $M^{\agr}$ has a linear resolution over
$R^{\agr}$ and then $\lin^{R}(\mathbf{F})$ is a minimal graded free
resolution of $M^{\agr}$.  The notion of the linearity defect can be
defined, in the same manner, for graded modules over a standard
graded algebra $A$ over a field $k$. In \cite{H-i}, the authors
proved that if $\ld_A(k)<\infty$, then $\ld_A(k)=0$.
Motivated by this known result in the graded case,
the following natural question raised in \cite{H-i}.

\begin{question}
\label{loc1}  If $\ld_R(k)< \infty,$ does it follow that
$\ld_R(k)=0$?
\end{question}
 If $R^{\agr}$ is Cohen-Macaulay, $\c{S}$ega \cite{S} showed that the question
has positive answer in the case that $R$ is a complete intersection.
Also, she gave an affirmative answer when $\m^3=0$.
Another
positive answer to the question is given by  the author and Rossi
\cite{AR} when $R$ is of homogeneous type, that is
 $\dim_k\Tor^R_i(k,k)=\dim_k
\Tor^{R^{\agr}}_i(k,k)$ for all $i$.\\
\indent In this paper we show that this problem has an affirmative
answer when $\m^4=0$. The proof relies on  the existence of a DG
algebra structure of a minimal free resolution of residue field
$k$.
\section{Preliminaries and the main result}
\indent \c{S}ega provided an interpretation of linearity defect in
term of vanishing of special maps. For each $n\geq 0$ and $i\geq0$
we consider the map
$$\upsilon^{n}_i(M):\Tor_{i}^R(M,R/\m^{n+1})\rightarrow
\Tor_{i}^R(M,R/\m^{n})$$ induced by the natural surjection
$R/\m^{n+1}\rightarrow R/\m^n$.
 For simplicity,  we set
$\upsilon^{n}_i:=\upsilon^{n}_i(k)$ when $M=k$.

\begin{thm}\cite[Theorem 2.2]{S}\label{S}
Let $M$ be a finitely generated  $R$-module and $d$ be an integer.
Then the following conditions are equivalent.
\begin{enumerate}
\item $\ld_R(M)\leq d;$
\item $\upsilon^{n}_i(M)=0$
for all $i\geq d+1$ and all $n\geq 0$.
\end{enumerate}
\end{thm}

\begin{rem}\label{first}
Let  $i\geq 0$. Assume that $\mathbf{F}$ is a minimal free resolution
of a finitely generated $R$-module $M$. Then by \cite[2.3
$(2^{'})$]{S}, the following statements are equivalent.
\begin{enumerate}

\item  $\upsilon^{1}_i(M)=0;$
\item if
$x\in F_{i}$ satisfies $\partial_{i}(x)\in \m^2F_{i-1}$, then $x\in
\m F_{i}$.
\end{enumerate}
\end{rem}

\indent  Let $S$ be a unitary commutative ring. Given an $S$-complex
$\mathbf{C}$, we write $|c|=i$ (the homological degree of $c$) when
$c\in\mathrm{C}_i$. When we write $c\in \mathbf{C}$ we mean $c\in C_i$ for some $i$. A (graded commutative) \textit{DG algebra} over $S$ is a
non-negative $S$-complex $(\mathbf{D},\partial)$ with  a morphism of
complexes called the product
\begin{eqnarray*}
\mu^{\mathbf{D}}: &\mathbf{D}\otimes_S\mathbf{D}\rightarrow  \mathbf{D}& \\
&a \otimes b\mapsto  ab&
\end{eqnarray*}
satisfying the following properties:
\begin{enumerate}
\item[$(i)$] unital: there is an element $1\in\mathrm{D}_0$ such that $1a=a1=a$ for
$a\in\mathbf{D} $;
\item[$(ii)$] associative: $a(ba)=(ab)c$ for all $a,b,c\in \mathbf{D}$;
\item[$(iii)$] graded commutative: $ab=(-1)^{|a||b|}ba\in D_{|a|+|b|}$ for
all $a,b\in \mathbf{D}$ and $a^2=0$ when $|a|$ is odd.
\end{enumerate}
The fact that $\mu$ is a morphism of complexes is expressed by the \textit{Leibniz rule}:
\[\partial(ab)=\partial(a)b+(-1)^{|a|}a\partial(b)\]
For more information on DG algebras  we refer to \cite{Av}.
\begin{rem}\label{homology}
If $(\mathbf{D},\partial)$ is a DG algebra over $S$. Using Leibniz
rule, one can see that the subcomplex of cycles
$Z(\mathbf{D})$ is a DG subalgebra of $\mathbf{D}$ and the
boundaries $B(\mathbf{D})$ is a DG ideal of $Z(\mathbf{D})$. Thus
the product on $\mathbf{D}$ induces a product on the homology
$H(\mathbf{D})=Z(\mathbf{D})/B(\mathbf{D})$. In particular,
$\oplus_{i\geq 0} H_n(\mathbf{D})$ is a graded module over
commutative ring $H_0(\mathbf{D})$.
\end{rem}
Tate constructed a DG algebra (free) resolution of $k$. Furthermore,
 such a  resolution can be chosen to be  minimal, see \cite[Theorem
 6.3.5]{Av}, which we refer a minimal Tate resolution of
$k$ over $R$.\\
\indent The following lemma shows that the linear part of a minimal
Tate resolution of $k$ inherits a DG algebra structure from that of
the resolution.

\begin{lem}\label{dg}
 Let $(\mathbf{F},\partial)$ be a minimal Tate resolution of $k$.
Then $\lin^R(\mathbf{F})$ has a DG algebra structure induced by that
of $\mathbf{F}$.
\end{lem}
\begin{proof}
Let  $\mu^{\mathbf{F}}:\mathbf{F}\otimes_R\mathbf{F}\rightarrow
\mathbf{F}$ be a morphism of complexes which defines the product on
$\mathbf{F}$. Set $S=R^{\agr}$. Since  $\mathbf{F}$ is minimal we
see that  $\mathbf{F}\otimes_R\mathbf{F}$ is a minimal complex as
well. Hence  the morphism induces a morphism of graded $S$-complexes
$\lin^R(\mu^{\mathbf{F}}):
\lin^R(\mathbf{F}\otimes_R\mathbf{F})\rightarrow \lin^R(\mathbf{F})$
such that  if  $i,n\geq 0$  and $x^*$ is the image of an element $x\in
\m^i(\mathbf{F}\otimes_R\mathbf{F})_n $ in $\m^i(\mathbf{F}\otimes_R\mathbf{F})_n/\m^{i+1}(\mathbf{F}\otimes_R\mathbf{F})_n  $, then $\lin^R(\mu^{\mathbf{F}})$ maps $x^*$ into
the image of $\mu(x)$ in $\m^iF_n/\m^{i+1}F_{n}$. There is also a natural
isomorphism of graded $S$-complexes
$\lambda:\lin^R(\mathbf{F})\otimes_S \lin^R(\mathbf{F})\rightarrow
\lin^R(\mathbf{F}\otimes_R\mathbf{F})$ such that if $i,j,n,m\geq 0$ and $x\in \m^{i}F_n$ and
$y\in \m^{j}F_m$ with images $x^*$ in $\m^{i}F_n/\m^{i+1}F_n $ and $y^*$ in $ \m^{j}F_m/ \m^{j+1}F_m$ respectively, then $\lambda$
 maps $x^*\otimes y^*$ into the image of $x\otimes y$ in
$\m^{i+j}(\mathbf{F}\otimes_R\mathbf{F})_{n+m}/\m^{i+j+1}(\mathbf{F}\otimes_R\mathbf{F})_{n+m}$ , see \cite[Lemma 2.7]{I-T}. Now, define (the
product)
\[ \mu^{\lin^R(\mathbf{F})}: \lin^R(\mathbf{F})\otimes_S \lin^R(\mathbf{F})\rightarrow
\lin^R(\mathbf{F})\] as the composition $\lin^R(\mu^{\mathbf{F}})
\circ\lambda $. Since $\mu$ satisfies conditions $(i),(ii),(iii)$ of
the definition of DG algebras, one can see that
$\mu^{\lin^R(\mathbf{F})}$ satisfies the same  properties as
well.
 Therefore the linear part of $\mathbf{F}$ is a DG algebra over
$S$ augmented to $k$.
\end{proof}

Let $\m^{*}$ denote the homogeneous maximal ideal of $R^{\agr}$. The
 following is a direct consequence of the above lemma.
\begin{cor}\label{ann}
 If $\mathbf{F}$ is a minimal free resolution of $k$, then
$\m^{*}\Ho_n(\lin^R(\mathbf{F}))=0$ for all $n$.
\end{cor}
\begin{proof}
The assertion follows from  Remark \ref{homology} with considering
the fact that $\Ho_0(\lin^R(\mathbf{F}))=R^{\agr}/\m^{*}$.
\end{proof}
In what will follow,  let $(\mathbf{F},\partial)$ be a minimal free
resolution of residue field $k$ with differential map $\partial$.
The differential map of $\lin^R(\mathbf{F})$  which is induced by
$\partial$ will be denoted by $\partial^*$.
 We recall that
$\lin^R(\mathbf{F})_n=F_n^{\agr}(-n)$. For  any  $ i,n\geq 0$ and $x
\in \m^iF_n$, $\partial^*$ maps $x+\m^{i+1}F_n$, the image of $x$ in
$\m^iF_n/\m^{i+1}F_n$, into the image of $\partial(x)$ in
$\m^{i+1}F_{n-1}/\m^{i+2}F_{n-1}$ that is $\partial(x)+
\m^{i+2}F_{n-1}$.
\begin{prop}\label{one}
Let $d$ be an integer. If
$\ld_R(k)\leq d$, then the following hold.
\begin{enumerate}
\item $\upsilon^{1}_d=0$.
\item
$\m^{*}\Ker \partial_{d}^{*}=\m^{*} \Image \partial_{d+1}^{*}$.

\end{enumerate}
\end{prop}
\begin{proof}
 For the simplicity, we set $\mathrm{Z}=\Ker \partial_{d}^{*}$ and
$\mathrm{B}= \Image \partial_{d+1}^{*}$.\\
\indent $(1)$ If $\upsilon^{1}_d\neq 0$, then there exists an
element $e\in F_{d}\setminus \m F_{d}$ such that $\partial_{d}(e)\in
\m^2F_{d-1}$, by \ref{first} . Let $e^*$ be the image of $e$ in the quotient module $F_d/\m F_d$. Then
$\partial_{d}^{*}(e^*)=0$ and so $e^*$ is a cycle in $\lin^R(\mathbf{F})$. Applying
\ref{ann}, we have $\m^{*}\mathrm{Z}\subseteq \mathrm{B}$ and
therefore $\m^{*}e^*\subseteq \mathrm{B}$. As $e^*$ is an
element of a basis of the free module $F^{\agr}_d(-d)$  and
$\mathrm{B}\subseteq \m^{*}F^{\agr}_{d}(-d)$, it is straightforward
to see that $\m^{*}e^*$ is a direct summand of $\mathrm{B}$. By
the hypothesis, $\mathrm{B}$ has a linear resolution. This implies
that the same property holds for $\m^{*}e^*$. Therefore $k$ has
a linear resolution over $R^{\agr}$ and consequently
$\lin^R(\mathbf{F})$ is acyclic. Hence $\upsilon^{1}_d= 0$ and this
is a
contradiction. \\
\indent For $(2)$, it is enough to show that
$\m^{*}\mathrm{Z}\subseteq \m^{*}\mathrm{B}$. First we claim that
$\m^{*}\mathrm{Z}$ is generated in degree at least $d+2$ and
$\m^{*}\mathrm{Z}\subseteq \mathrm{B}$. Indeed since
$\upsilon^{1}_d= 0$, applying Remark \ref{first}, one has
$\mathrm{Z}\subseteq \m^{*}F^{\agr}_d(-d)$ and consequently
$\m^{*}\mathrm{Z}$ is generated in degree at least $d+2$. The second
part of the claim follows from Corollary \ref{ann}.\\
\indent On the other hand, $\mathrm{B}$ has a linear resolution, by
the hypothesis. Hence $\mathrm{B}$ is generated by elements of
degree $d+1$ and then all its elements of degree at least $d+2$
contained in $\m^{*}\mathrm{B}$. Now, putting these two considerations
together, we get $\m^{*}\mathrm{Z}\subseteq \m^{*}\mathrm{B}$.
\end{proof}

Now, we are ready to prove our main result.
\begin{thm}
Assume that $R$ is Artinian with $\m^4=0$. If $\ld_R(k)<\infty$,
then $\ld_R(k)=0.$
\end{thm}
\begin{proof}
Let $d$ be a non-negative  integer and $\ld_R(k)\leq d$. We prove by
descending induction on $d$. The case where $d=0$ is clear. Let
$d>0$. Applying  Proposition \ref{one}, we have $\upsilon^{1}_d=0$.
Since $\m^4=0$, it follows from \cite[Theorem 7.1]{S} that
$\upsilon^{2}_d=0$. Again since $\m^4=0$ obviously
$\upsilon^{i}_d=0$ for all $i\geq 3$, by the definition  of the map
$\upsilon^{i}_d$. Therefore,  from \ref{S} we get $\ld_R(k)\leq
d-1$. This completes the induction and finishes the proof.
\end{proof}
\vskip 1 cm
\textbf{Acknowledgments.}
 The author would like to express great thanks to the referee for valuable comments
and suggestions which have improved the exposition of this paper. This research
was in part supported by a grant from IPM (No. 94130028). This work was also
jointly supported by the Iran National Science Foundation (INSF) and Alzahra
University grant No. 95001343.

\end{document}